\documentclass[11pt, reqno]{amsart}
\usepackage{euscript,amscd,amsgen,amsfonts,amssymb,latexsym}

\newcommand{\Gateaux}{Gateaux}
\newcommand{\eqdef}{\stackrel{\scriptscriptstyle\rm def}{=}}
\usepackage{todonotes}
\newtheorem{theorem}{Theorem}
\newtheorem*{thmM}{Main Theorem}

\newtheorem{corollary}{Corollary}

\newtheorem{lemma}{Lemma}
\newtheorem{remark}{Remark}

\DeclareMathOperator{\card}{card}

\renewcommand{\epsilon}{\varepsilon}

\newcommand{\cM}{\EuScript{M}}
\newcommand{\cQ}{\mathcal{Q}}
\newcommand{\cP}{\EuScript{P}}

\newcommand{\cH}{\EuScript{H}}
\newcommand{\cRR}{\EuScript{R}}

\newcommand{\bR}{{\mathbb R}}

\newcommand{\bZ}{{\mathbb Z}}
\newcommand{\bN}{{\mathbb N}}

\newcommand{\cA}{{\mathcal A}}

\newcommand{\cL}{{\mathcal L}}

\def\dist{\text{{\rm dist}}}

\DeclareMathSymbol{\varnothing}{\mathord}{AMSb}{"3F}
\renewcommand{\emptyset}{\varnothing}

\title{Multiple phase transitions on compact symbolic systems}

\author{Tamara Kucherenko}\address{Department of Mathematics,
The City College of New York, New York, NY, 10031, USA}\email{tkucherenko@ccny.cuny.edu}

\author{Anthony Quas}\address{ Department of Mathematics and Statistics, University of Victoria, Victoria, BC
Canada}\email{aquas@uvic.ca}

\author{Christian Wolf}\address{Department of Mathematics, The City College of
New York, New York, NY, 10031, USA}\email{cwolf@ccny.cuny.edu}

\thanks{T.K. is supported by grants from the Simons Foundation \#430032 and from the PSC-CUNY TRADA-48-19. }
\thanks{A.Q. is supported by a grant from NSERC}
\thanks{C.W. was supported by  grants from  the Simons Foundation (\#637594 to Christian Wolf) and from the PSC-CUNY (TRADB-51-63715).}

\begin{document}

\begin{abstract}
Let $\phi:X\to \bR$ be a continuous potential associated with a symbolic 
dynamical system $T:X\to X$ over a finite alphabet.  Introducing a parameter 
$\beta>0$ (interpreted as the inverse temperature) we study the regularity of
the pressure function $\beta\mapsto P_{\rm top}(\beta\phi)$ on an interval 
$[\alpha,\infty)$ with $\alpha>0$. We say that $\phi$ has a phase transition at 
$\beta_0$ if the pressure function $P_{\rm top}(\beta\phi)$ is not differentiable 
at $\beta_0$. This is equivalent to the condition that the potential $\beta_0\phi$ 
has two (ergodic) equilibrium states with distinct entropies. For any $\alpha>0$ 
and any increasing sequence of real numbers $(\beta_n)$ contained in 
$[\alpha,\infty)$, we construct a potential $\phi$ whose phase transitions in 
$[\alpha,\infty)$ occur precisely at the $\beta_n$'s. In particular, we obtain a 
potential which has a countably infinite set of phase transitions. 
\end{abstract}





\subjclass[2000]{}
\maketitle
\section{Introduction} Broadly speaking a phase transition refers to a qualitative 
change of the statistical properties of a dynamical system. The precise definition 
of this notion differs depending on which settings and properties one studies. A phase 
transition might mean co-existence of several equilibrium states resulting from some 
optimization \cite{CET,E,G}, lack of the \Gateaux{} differentiability of the pressure 
functional \cite{CR,DE,IRV,Ph}, or loss of analyticity of the pressure with respect to 
external physical parameters such as temperature \cite{IT,JOP,S1}. The latter gives 
rise to further differentiation between the first-order, second-order, or higher order phase transitions.

We briefly note the relationship between the regularity of the pressure and coexistence 
of several equilibria. We refer to Section \ref{ThermodynamicFormalism} for the definitions 
and a formal discussion. It was shown by Walters \cite{W1} that the \Gateaux{} differentiability 
at $\phi$ of the pressure functional acting on the space of continuous potentials is equivalent 
to the uniqueness of the equilibrium state for $\phi$. On the other hand, non-differentiability 
of the pressure in the direction of $\phi$ necessarily implies coexistence of several equilibrium 
states. Recently, Leplaideur \cite{Lep} discovered the surprising fact that the converse to 
this statement is not true. He provided an example of a continuous potential $\phi$ defined 
on a mixing subshift of finite type such that the pressure is analytic in the direction of $\phi$, 
but uniqueness of equilibrium states fails. Moreover, in Leplaideur's example the uniqueness 
of equilibrium states fails for two  distinct inverse temperature values.

The existence of phase transitions has also been established for parabolic systems 
and the geometric potential, see, e.g., \cite{ADU,BI,Lo,UW}. Roughly speaking in these 
examples the \emph{degree of parabolicity} near the parabolic point(s) determines 
whether the second equilibrium state is finite or $\sigma$-finite.

A number of related questions have been studied in the statistical physics literature. 
Miekisz \cite{M} starts from the well-known Robinson two-dimensional shift of finite 
type that admits only aperiodic configurations. These configurations have a highly 
regular structure, and are considered to be a quasi-crystal. Miekisz's model allows 
``forbidden" configurations to appear with a local energy cost. There is a conjectural 
picture in the statistical physics literature which would imply that at finite temperatures, 
the equilibrium states exhibit global periodicity with local fluctuations, while at zero 
temperature the equilibrium measure is supported on quasi-crystal states. In this picture, 
there is a sequence of phase transitions as the temperature is reduced at which the 
global period increases. In support of this conjecture, Miekisz establishes an increasing 
sequence of lower bounds on the global period as the temperature is reduced to zero. 
It should be pointed out that the paper does not demonstrate that the quasi-crystal 
state is not attained at positive temperature, but rather establishes that conditional 
on the positive temperature states not being quasi-crystalline, there must be an 
infinite sequence of phase transitions.

Another related model due to van Enter and Shlosman \cite{ES} deals with a model in 
dimension 2 or higher with a continuous alphabet (the circle). They describe a nearest 
neighbour potential (the ``Seuss model", named after the Cat in the Hat) with an 
infinite sequence of first order phase transitions in this setting.

In this note we are concerned with the first-order phase transitions of the pressure 
function with respect to a parameter regarded as the inverse temperature. We consider 
a continuous potential $\phi:X\to \bR$ associated with a symbolic dynamical system 
$(X,T)$ over a finite alphabet.  Given a positive real number $\beta$, we study the 
regularity of the pressure function $\beta\mapsto P_{\rm top}(\beta\phi)$. We say 
$\phi$ has a phase transition at $\beta_0$ if the pressure function 
$\beta\mapsto P_{\rm top}(\beta\phi)$ is not differentiable at $\beta_0$. 
This is equivalent to the condition that the potential $\beta_0\phi$ has two 
(ergodic) equilibrium states with distinct entropies (see Section 
\ref{ThermodynamicFormalism} for details).

It is a classical result due to Ruelle \cite{Ru1,Ru2} that if $X$ is a transitive subshift 
of finite type then the pressure functional $P_{\rm top}$ acts real analytically on 
the space of H\"older continuous potentials, that is, for all H\"older continuous 
$\phi,\psi:X\to \bR$ we have that $\beta\mapsto P_{\rm top}(\phi+\beta\psi)$ 
is analytic in a neighborhood of $0$. This immediately implies the uniqueness of 
equilibrium states for H\"older continuous potential, which is referred to as 
``lack of phase transitions," \cite{E}. Therefore, in order to allow the possibility 
of phase transitions (i.e. the occurrence of distinct equilibrium states) one needs 
to consider potential functions that are merely continuous.

Although phase transitions have been studied for many classes of dynamical 
systems, to the best of our knowledge this is the first family of examples in the 
one-dimensional symbolic setting with more than two phase transitions.
In this paper we develop a method to explicitly construct a continuous potential 
with any finite number of first order phase transitions in any given interval $[\alpha,\infty)$ 
occurring at any sequence of predetermined points. We are able to go even further. 
Note that the convexity of the pressure implies that a continuous potential $\phi$ 
has at most countably many phase transitions. We show that the case of infinitely 
many phase transitions can indeed be realized. In the following statement we 
summarize our results given by Theorem \ref{main}, Corollary \ref{FiniteCase}, 
Corollary \ref{Quasi-crystal}, Corollary \ref{cor123} and Remark \ref{NoMoreTransitions}.

\begin{thmM}Let $T:X\to X$ be the two-sided full shift, $\alpha$ be any positive number 
and $(\beta_n)$ be a strictly increasing (finite or infinite) sequence in $[\alpha,\infty)$. 
Then there exists a continuous potential $\phi:X\to\bR$ such that the following holds:
\begin{itemize}
\item[(i)] When $\beta\ge\alpha$ the potential $\phi$ has a phase transition at 
$\beta$ if and only if $\beta=\beta_n$ for some $n\in \bN$;
\item[(ii)] If $\lim\limits_{n\to\infty} \beta_n=\beta_\infty<\infty$, then  the family 
of equilibrium states of $\beta\phi$ is constant for all $\beta\ge\beta_\infty$.
\end{itemize}
\end{thmM}

While we are able to completely control the function $\beta\mapsto P(\beta\phi)$ 
on the interval $[\alpha,\infty)$, we do not have a full picture of the behaviour for $\beta$ 
values in the range $[0,\alpha)$. In the range $[\alpha,\infty)$, the potential drops so 
sharply off certain sets $X_n$ that equilibrium states are forced to live on the $X_n$'s. 
For $\beta$ below $\alpha$, the ``cost" incurred by leaving $\bigcup X_n$ is not 
sufficiently high to prevent equilibrium states having support outside the union. 
We conjecture that there exists a $\beta_*$ such that for $\beta<\beta_*$, the 
equilibrium state for $\beta\phi$ is fully supported, and indeed in this region, we 
expect that the pressure depends analytically  on $\beta$.

Of particular interest is the behavior of the pressure function 
$\beta\mapsto P_{\rm top}(\beta\phi)$ as $\beta\to\infty$. A simple argument 
shows that $P_{\rm top}(\beta\phi)$ has an asymptote of the form $a\beta+b$.  
Taking finitely many $(\beta_n)_{n=1}^{N}$ in part (i) of the Main Theorem we see
that $P_{\rm top}(\beta\phi)$ reaches its asymptote at $\beta=\beta_N$. Hence, we 
have an ultimate phase transition at $\beta=\beta_N$. Physically, this means that 
for some positive temperature $1/\beta_N$, the systems reaches its ground state 
which is the unique measure of maximal entropy of a certain subshift of  $X$ in our 
construction, and then ceases to change. This phenomenon is often referred  to as a 
freezing phase transition at which the system reaches a ground state. We refer to 
\cite{BL,BL2} and the references therein for details about freezing phase transitions.

We note that in the situation of the Main Theorem part (ii) with infinitely many 
$\beta_n$ values we also obtain a freezing phase transition. Namely, for all 
$\beta\ge \beta_{\infty}$ there are precisely two ergodic equilibrium states both 
of which are  fixed point measures and each equilibrium state is a convex 
combination of these fixed point measures.
In particular, the set of equilibrium states of $\beta\phi$ does not change 
anymore when the temperature $1/\beta$ is lowered.

Finally, we mention that the case of finite alphabet shift maps crucially differs 
from that of  countable alphabet symbolic systems. Indeed, for countable Markov 
shifts Sarig established several new phenomena that are associated with the
lack of analyticity of the pressure function. This includes positive Lebesgue measure 
non-analyticity points of the pressure function, which are associated with the 
existence of multiple equilibrium states and/or intervals of intermittent behavior, 
i.e.{} an interval of $\beta$'s  with an infinite conservative equilibrium state. We 
refer to \cite{S1,S2} for details.

We thank Aernout van Enter for helpful information about the statistical physics literature. 

\section{Preliminaries}

\subsection{Thermodynamic Formalism.}\label{ThermodynamicFormalism}
Let  $T:X\to X$ be a homeomorphism on a compact metric space $X$, and denote
by $\cM=\cM_T$ the set of all $T$-invariant probability measures on $X$ endowed 
with the weak$^\ast$ topology. This makes $\cM$ a compact convex metrizable 
topological space. Further, let $\cM^e=\cM^e_T$ be the subset of ergodic measures. 
For $\mu\in \cM$ the
measure-theoretic entropy of $\mu$, denoted by $h_\mu(T)$, is defined as follows.
Let $\cP$ be a countable measurable partition of $X$. We define the entropy
of the partition $\cP$ with respect to $\mu$
\begin{equation}\label{PartitionEntropy}
  H_\mu(\cP)=\sum_{P\in\cP}-\mu(P)\log \mu(P).
\end{equation}
Here we interpret $0\log 0$ as $0$.
We denote by $\bigvee_{i=m}^nT^i(\cP)$ the coarsest common refinement of the partitions
$T^m(\cP),...,T^n(\cP)$. Then the entropy of $\mu$ with respect to the partition $\cP$ and $T$ is given by

\begin{equation}\label{MetricEntropy}
  h_\mu(T,\cP)=\lim\limits_{n\to\infty}\frac1n H_\mu\left(\bigvee_{i=0}^nT^{-i}(\cP)\right)=
  \inf\frac1n H_\mu\left(\bigvee_{i=0}^nT^{-i}(\cP)\right)
\end{equation}
Finally, we define the measure-theoretic entropy of $\mu$ by
\begin{equation}\label{defmuent}
h_\mu(T)=\sup\left\{h_\mu(T,\cP): H_\mu(\cP)<\infty\right\}.
\end{equation}
We note that  definition \eqref{defmuent}  is equivalent to the frequently used 
definition in terms of finite partitions, see, e.g. \cite{Ke}. We say $\cP$ is a 
generating partition if $\bigvee_{i=-\infty}^\infty T^i(\cP)$ coincides with the 
original $\sigma$-algebra. For a generating partition $\cP$ with finite entropy
$H_\mu(\cP)$ we have $h_\mu(T)=H_\mu(T,\cP)$.

We may define the \emph{topological entropy} of the system $(X,T)$ via the variational principle:
\begin{equation}\label{TopEntropyf}
  h_{\rm top}(T)=h_{\rm top}(X,T)=\sup\{h_\mu(T):\mu\in\cM\}=\sup\{h_\mu(T):\mu\in\cM^e\}.
\end{equation}
If a  measure $\mu$ realizes the supremum in \eqref{TopEntropyf} we say that $\mu$ is a
\emph{measure of maximal entropy}. Note that if the entropy map $\mu\mapsto h_\mu(T)$
is upper semi-continuous on $\cM$ then there exists at least one measure of
maximal entropy. This holds, in particular, for symbolic systems over finite alphabets 
which are considered in this paper.

The \emph{topological pressure} (with respect to $T$) is a mapping
$ P_{\rm top}\colon C(X,\bR)\to \bR\cup\{\infty\}$ which is convex, $1$-Lipschitz continuous and  satisfies
\begin{equation}\label{VarPrinciple}
P_{\rm top}(\phi)=
\sup_{\mu\in \cM} \left(h_\mu(T)+\int_X \phi\,d\mu\right).
\end{equation}
For a $T$-invariant set $Y\subset X$ we frequently use the notation $P_{\rm top}(Y,\phi)$ 
for the topological pressure with respect to $T\vert_Y$ of the potential $\phi\vert_Y$.
As in the case of topological entropy, the supremum in (\ref{VarPrinciple}) can be replaced by
the supremum taken only over all $\mu\in\cM^e$. If there exists a measure
$\mu\in \cM$ at which the supremum in \eqref{VarPrinciple} is
attained it is called an \emph{equilibrium state} (or also equilibrium measure) of the potential $\phi$.

It turns out that there is a connection between the structure of the set of equilibrium
states of a potential $\phi$ and \Gateaux{} differentiability of the pressure mapping
at $\phi$. As for any continuous real-valued map on a Banach space, the pressure is \Gateaux{}
differentiable at $\phi$ if and only if the tangent functional to $P_{\rm top}$ at $\phi$ is
unique. Since every equilibrium state of $\phi$ is a tangent functional to $P_{\rm top}$ at
$\phi$, non-uniqueness of equilibrium states immediately implies
non-differentiability of the pressure at $\phi$. We further note that if $T$ has an 
upper semi-continuous entropy map
then $\mu$ is a tangent functional  if and only if $\mu$ is an equilibrium state \cite{Ke}.

We are interested in the differentiability of the pressure function
$\beta\mapsto P(\beta)\eqdef P_{\rm top}(\beta\phi)$ for a fixed potential $\phi$.
The convexity of the topological pressure implies that left and right derivatives for
$P(\beta)$, denoted $d^+P(\beta)$ and $d^-P(\beta)$ respectively,
exist at every point $\beta$. Moreover, we have (see, e.g., \cite{W1}):
\begin{align*}
  d^+P(\beta) & =\sup\left\{\int\phi \,d\mu:\,\,
  \mu \text{ is a tangent functional to $P$ at }\beta\phi\right\} \\
  d^-P(\beta) & =\inf\left\{\int\phi \,d\mu:\,\,
  \mu \text{ is a tangent functional to $P$ at }\beta\phi\right\}.
\end{align*}
Hence, simply having two distinct equilibrium states for a potential
$\beta_0\phi$ does not guarantee non-differentiability of
$P(\beta)$ at $\beta_0$. We must, in addition, require
that the integrals of $\phi$ with respect to the equilibrium states differ,
or equivalently for $\beta_0\ne 0$, that there are (ergodic) equilibrium states of 
$\beta_0\phi$ with different entropies.

\subsection{Symbolic Spaces.}
Let $d\in \bN$ and let $\cA=\{0,\dots,d-1\}$ be a finite alphabet with $d$
symbols. The (two-sided) shift space $X$ on the alphabet $\cA$ is the set
of all bi-infinite sequences $x=(x_n)_{n=-\infty}^\infty$ where $x_n\in \cA$
for all $n\in \bZ$.  We endow $X$ with the Tychonov product topology
which makes $X$ a compact metrizable space. It is easy to see that
\begin{equation}\label{defmet}
d(x,y)=2^{-\inf\{|n|:\  x_n\not=y_n\}}
\end{equation}
defines a metric which induces the Tychonov product topology on $X$.
Note that $d$ is actually an ultrametric, i.e. it satisfies the strong triangle inequality 
$d(x,y)\le \max\{d(x,z),\, d(z,y)\}$ for all $x,y,z\in X$.

The shift map $T:X\to X$ given by $T(x)_n=x_{n+1}$ is a
homeomorphism on $X$. For a word $w=w_0\cdots w_{n-1}\in \cA^n$,
we denote by $[w]=[w_0\cdots w_{n-1}]=\{x\in X: x_0=w_0,\dots,
x_{n-1}=w_{n-1}\}$ the cylinder generated by $w$.
For any two words $w=w_0\cdots w_{n-1}$ with $n\in \bN$ and
$v=v_0\cdots v_{m-1}$ with $m\in \bN\cup\{\infty\}$ we denote by
$wv$ their concatenation  $wv=w_0\cdots w_{n-1}v_0\cdots v_{m-1}$.

If $Y\subset X$ is a non-empty closed $T$-invariant set we  say
that $T\vert_Y$ is a sub-shift. For a subshift $Y\subset X$ we denote by $\cL_n(Y)$
the set of all admissible words in $Y$ of length $n$, i.e.
$$
\cL_n(Y)=\{w\in \cA^n:
[w]\cap Y\ne\emptyset\}.
$$
Then $\cL(Y)=\bigcup_{n=1}^\infty \cL_n(Y)$ is referred to as the
\emph{language} of $Y$. One way to create various subshifts is through
coded systems. We say that $(Y,T)$ is a \emph{coded system} if $Y$
it is the closure of the set of sequences obtained by freely concatenating
the words in a set $W \subset \cL(X)$. In this case $W$ is called the \emph{generating set} of Y.

We conclude this section with a remark that the topological entropy of a
subshift $Y$ can be computed as the logarithmic growth of the number
of words length $n$ in the language of $Y$. More precisely,
\begin{equation}\label{TopEntropy}
  h_{\rm top}(Y,T)=\lim\limits_{n\to\infty}\frac{\log\card\cL_n(Y)}{n}.
\end{equation}
Here and for the reminder of the paper all logarithms are taken to base 2.

\section{Proof of the Main Theorem}

We deal first with the case of a countably infinite number of phase transitions. 
The reduction to the case of finitely many phase transitions is done in Corollary \ref{FiniteCase}.
\begin{theorem}\label{main}
  Let $X$ be a two-sided full shift on two symbols. Then for any positive real number $\alpha$ and any
  strictly increasing sequence $(\beta_n)_{n=1}^\infty$ in $[\alpha,\infty)$ there is a continuous potential
  $\phi:X\to\bR$ such that the pressure function $\beta\mapsto
  P(\beta\phi)$ on $[\alpha,\infty)$ is not differentiable exactly
  at points $\beta_n,\, n\in\bN$.
\end{theorem}
We build a potential $\varphi$ on the full shift $(X,T)$ with alphabet $\cA=\{0,1\}$ whose equilibrium states
move among a sequence of disjointly supported subshifts. Specifically we take
$(X_n)$ to be a sequence of disjoint proper shifts of finite type that are sub-systems
of $X$. We also define $X_\infty$ to be the set of accumulation points of $\bigcup X_n$.
We fix a positive strictly increasing sequence $(\beta_n)_{n=1}^\infty$ and 
find a corresponding sequence of values $(c_n)_{n=1}^\infty$.
The potential we build is then initially specified as a constant $c_n$ on each $X_n$,
and $c=\lim c_n$ on $X_\infty$. The potential then drops sharply for points in $X$
outside $\bigcup X_n$ as a means of forcing the equilibrium measures at all values of
the inverse temperature to be supported on $\bigcup X_n$. Clearly, care needs to be
taken to ensure that the resulting potential is continuous.
Provided that the drop-off is sufficiently steep, we obtain
$P_{\rm top}(X,\beta\phi)=\max_n P_{\rm top}(X_n,\beta\phi)$.
The phase transitions occur at those values of $\beta$ for which
$P_{\rm top}(\beta\phi)=P_{\rm top}(X_n,\beta\phi)=P_{\rm top}(X_{n+1},\beta\phi)$.

For $n\in\bN$ consider the subshift $X_n$ generated by
$W_n=\{(\underbrace{00...01}_{2^n}),(\underbrace{11...10}_{2^n})\}$.
We denote the topological entropy of $X_n$ by $h_n$ and the measure of
maximal entropy for $X_n$ by $\mu_n$. It is easy to see that $h_n=2^{-n}$ and
the subshifts $(X_n)$ are pairwise disjoint.

The next lemma provides a way to select the appropriate values of the
potential on the subshifts $X_n$ for a given set of discontinuity points $(\beta_n)$.

\begin{lemma}\label{Choice of c_n}
Suppose $(\beta_n)_{n=1}^\infty$ is a strictly increasing sequence of positive 
real numbers. Then there exists a strictly increasing sequence $(c_n)_{n=1}^\infty$ with
$\lim\limits_{n\to\infty} c_n=c$ such that for any $n\in\bN$ the following holds.
\begin{enumerate}
 \item $h_{n}+\beta_n c_{n}=h_{n+1}+\beta_n c_{n+1}$.
 \item For any $\beta\in[\beta_{n},\beta_{n+1}]$ and any $k\in\bN$ we have 
  $h_{k}+\beta c_{k}\le h_{n+1}+\beta c_{n+1}$
  with strict inequality provided that $k\not\in\{n,n+1\}$. 
  Moreover, for $\beta\in[0,\beta_1]$ we have $h_k+\beta c_k\le h_1+\beta c_1$ for $k\in\bN$.
 \item In addition, $\beta c\le h_{n}+\beta c_{n}$ for any positive $\beta\le\beta_n$.
\end{enumerate}
\end{lemma}
\begin{proof}
  We may take
  \begin{equation}\label{Def_c_n}
    c_n=-\sum_{j=n}^{\infty}\frac{1}{2^{j+1}\beta_j}.
  \end{equation}
We immediately see that $c_{n+1}-c_{n}=\frac{1}{2^{n+1}\beta_n}>0$
and hence $(c_n)_n$ is strictly increasing. Also, $\lim\limits_{n\to\infty} c_n=0$
since $(c_n)_n$ is the sequence of tail sums of a convergent series. The first assertion of the
lemma is equivalent to $\beta_n(c_{n+1}-c_n)=h_n-h_{n+1}$, which is true because
$$
\beta_n(c_{n+1}-c_n)=\beta_n\frac{1}{2^{n+1}\beta_n}=
\frac{1}{2^{n+1}}=h_{n}-h_{n+1}.
$$

To check the second assertion we first show that for any $k\not\in\{n,n+1\}$ 
we have $h_{k}+\beta_n c_{k}< h_{n}+\beta_n c_{n}$. We consider two cases $k<n$ and $k>n+1$.
If $k<n$ then since $(\beta_n)$ is increasing, we get
$$
\beta_n(c_{n}-c_k)=\beta_n\sum_{j=k}^{n-1}\frac{1}{2^{j+1}\beta_j}>
\beta_n\sum_{j=k}^{n-1}\frac{1}{2^{j+1}\beta_n}=
\sum_{j=k}^{n-1}\frac{1}{2^{j+1}}=\frac{1}{2^k}-\frac{1}{2^n}=h_k-h_n.
$$
If $k>n+1$,
$$
\beta_n(c_{n}-c_k)=-\beta_n\sum_{j=n}^{k-1}\frac{1}{2^{j+1}\beta_j}<
-\beta_n\sum_{j=n}^{k-1}\frac{1}{2^{j+1}\beta_n}=
-\sum_{j=n}^{k-1}\frac{1}{2^{j+1}}=h_k-h_n.
$$
When $\beta\in [\beta_{n},\beta_{n+1}]$ we can write $\beta=t\beta_{n}+
(1-t)\beta_{n+1}$ for some $t\in[0,1]$. Using the above observation and the first
part of the lemma we obtain
\begin{align*}
  h_k+\beta c_k & =t(h_k+\beta_{n} c_k)+(1-t)(h_k+\beta_{n+1} c_k) \\
   & \le t(h_{n}+\beta_{n} c_{n})+(1-t)(h_{n+1}+\beta_{n+1} c_{n+1}) \\
   &=h_{n+1}+\beta c_{n+1}.
\end{align*}
Note that the inequality here is strict if $k\ne n$ and $k\ne n+1$. In the case when 
$\beta\in [0,\beta_1]$ we have 
$$
\beta (c_k-c_1)=
\sum_{j=1}^{k-1}\frac{\beta}{2^{j+1}\beta_j}\le\sum_{j=1}^{k-1}\frac{1}{2^{j+1}}
=\frac{1}{2}-\frac{1}{2^k}=h_1-h_k.
$$
Hence, $h_k+\beta c_k\le h_1+\beta c_1$, which completes the proof of part (2).

Finally, the third assertion follows from the facts that $c=0$, but whenever $0\le\beta\le\beta_n$ we have
$$
h_n+\beta_nc_n=\frac{1}{2^n}-\beta_n\sum_{j=n}^{\infty}\frac{1}{2^{j+1}\beta_j}
\ge \frac{1}{2^n}-\sum_{j=n}^{\infty}\frac{1}{2^{j+1}}=0.
$$
\end{proof}

We note that the sequence $(c_n)_{n=1}^\infty$ which satisfies conditions (1)-(3) of Lemma
\ref{Choice of c_n} is determined by the sequence $(\beta_n)_{n=1}^\infty$ up to addition of a constant.

\begin{lemma}\label{Def_phi}
  Suppose $(\beta_n)_n$ and $(c_n)_n$  satisfy Lemma \ref{Choice of c_n}.
  Let $\alpha$ be a fixed real number with $0<\alpha<\beta_1$ and set 
  $\phi(x)=\sup\{\phi_n(x):n\in\bN\}$, where
  $\phi_n:X\to\bR$ is defined by
  \begin{equation}\label{Def_phi_n}
 \phi_n(x)= c_n-\delta_{j},\text{ if }\,{\rm dist}(x, X_n)=\frac{1}{2^j}\quad
 \text{with}\,\,\delta_j=\frac{10+3\log j}{\alpha j}.
\end{equation}
Then
\begin{itemize}
  \item $\phi$ is continuous on $X$;
  \item $\phi|_{X_n}=c_n$;
  \item $P_{\rm top}(X_n,\beta_n\phi)=P_{\rm top}(X_{n+1},\beta_n\phi)$;
  \item If $k\ne n$ and $k\ne n+1$ then $P_{\rm top}(X_{k},\beta_n\phi)<P_{\rm top}(X_n,\beta_n\phi)$.
\end{itemize}
\end{lemma}
\begin{proof}
First notice that $(\delta_j)$ is decreasing to 0 and write $\delta_\infty=0$. We claim that
for any $n$, $|\varphi_n(x)-\varphi_n(y)|\le\delta_j$ if $d(x,y)\le2^{-j}$.
To see this, let $d(x,y)=2^{-j}$ and $d(x,X_n)=2^{-k}$ (where $k\in\bN\cup\{\infty\})$. Since $d$
is an ultrametric, $d(y,X_n)\le \max(2^{-j},2^{-k})$, so that $\varphi_n(y)\ge
c_n-\delta_{\min(j,k)}$. Since $\varphi_n(x)=c_n-\delta_k$, it follows that
$\varphi_n(y)\ge \varphi_n(x)-\delta_j$. Switching the role of $x$ and $y$, we see that
$|\varphi_n(y)-\varphi_n(x)|\le\delta_j$.
Hence the $\{\varphi_n\}$ are uniformly equicontinuous. Since they are uniformly
bounded, it follows that $\varphi(x)=\sup_n\varphi_n(x)$ is continuous.

Next, we show that $\phi|_{X_n}=c_n$ for all $n\in\bN$. If $x\in X_n$, then 
$\phi_n(x)=c_n$ and for $k<n$, $\phi_k(x)<c_k<c_n$.
For $k>n$, $d(x,X_k)\ge 2^{-2^n}$ (with equality only if $x_0$ is aligned
with the last symbol in a $W_n$ block), so that $\phi_k(x)\le c_k-\delta_{2^n}
<-10/(\alpha 2^n)$. Since $c_n>-\frac{1}{\alpha}\sum_{k=n+1}^\infty 2^{-k}=
-1/(\alpha 2^n)$, we see that $\phi(x)=\phi_n(x)=c_n$.

The other properties follow directly from Lemma \ref{Choice of c_n}.
\end{proof}

To establish Theorem \ref{main},
the main task is to prove that for all $\beta\ge\alpha$, $P_\text{top}(\beta\phi)=
\sup_n P_\text{top}(X_n,\beta\phi)$.
We do this by using the Variational Principle \eqref{VarPrinciple}.
The above lemma controls $h_\mu(T)+\beta_n\int \phi\,d\mu$ for ergodic measures supported on any
of the $X_k$. In the following we show that
for any ergodic measure $\mu$ that is not supported on $\bigcup X_k$  and any 
$\beta \in [\beta_{n-1},\beta_n]$ we have
$$
h_\mu(T)+\beta\int \phi\,d\mu \le h_{\mu_n}(T)+\beta\int\phi\,d\mu_n.
$$
We will make heavy use of a pin-sequence construction on a product space as 
outlined below. For a survey of techniques by which theorems about dynamical 
systems are proved via appending an auxiliary system we refer the reader to 
\cite{Qu}. One particular application similar to ours can be found in \cite{A}.

Let $Y=\overline{\bigcup X_n}=\bigcup X_n\cup X_\infty$, which is a closed invariant subset of $X$.
Consider an additional full shift $Z=\{0,1\}^{\bZ}$ and the product system
$X\times Z$ with map $\bar T(x,z)=(T x,T z)$.
Let $P$ be the collection of pairs $(x,z)$ with the following properties:
\begin{itemize}
\item If $i<j$, and $z_k=0$ for $i<k\le j$ (possibly with $z_i=1$) then $x_i\cdots x_j\in \cL(Y)$;
\item If $i<j$, $z_i=1$, $z_j=1$ then $x_i\cdots x_j\not\in\cL(Y)$.
\end{itemize}
We refer to the space $P$ as the \emph{pin-sequence space} and to the 1's in $z$ as \emph{pins}.
We note that $P$ is a subshift of $X\times Z$.

Given a point $x\in X$ we describe a construction of a
sequence $z\in Z$ such that $(x,z)\in P$ in the following way.
\begin{itemize}
  \item For $N\in\bN$ set $n_0^N=-N$
  \item Given $n_{k-1}^N$ let $n_k^N=\min\{i>n_{k-1}^N:\, (x_{n_{k-1}^N},\ldots,x_i)\notin \cL(Y)\}$
  \item Let $z^N\in Z$ have coordinates $z_i^N=
  \left\{
     \begin{aligned}
        1,\quad&\text{ if }\, i\in (n_k^N)_{k=0}^\infty\\
        0,\quad&\text{ otherwise }
     \end{aligned}
  \right.$
  \item We obtain a pin sequence of $x$ by taking any limit point of $(z^N)_N$
\end{itemize}

This is a way of greedily partitioning $x\in X$ into maximal legal blocks from $\cL(Y)$.
Note that for $x\in Y$ the pin sequence obtained by this construction is $z=(\bar{0})$.

Let $\mu$ be any ergodic invariant measure on $X$ with $\mu(Y)=0$.
Fix any $x'\in X$ which is generic for $\mu$, i.e.,
$$
\mu=\lim_{n\to\infty}\frac1n\sum_{j=0}^{n-1}\delta_{T^j(x')}.
$$
Take a pin sequence $z'$ corresponding to $x'$ and consider a limit point of
$\frac1n\sum_{j=0}^{n-1}\delta_{\bar T^j(x', z')}$. This is an invariant measure on
$P$, which we will denote by $\bar{\nu}$. Then $\bar{\nu}\circ\pi_X^{-1}=\mu$,
where $\pi_X:X\times Z \mapsto X$ is the projection map. Since $\mu(Y)=0$, 
there exists $n\in \bN$ and  $w\not\in \cL_n(Y)$
such that $\mu([w])>0$. The ergodicity of $\mu$ implies that $\mu$-a.e.\ $x$ contains infinitely
many $w$ blocks in its negative coordinates.
In $\bar\nu$-a.e. $(x,z)$, each such $w$ block must contain a pin, and given $x$ and the location of
one pin, the pins to the right are determined. It follows that for $\mu$-a.e.\ $x$, there are at most $n$
points $(x,z)$ in the pin sequence space. Hence the map $\pi_X$ is finite-to-one $\bar\nu$-almost everywhere
and it follows that $h_\mu(T)= h_{\bar{\nu}}(T)$.

We shall derive upper bounds for
$h_{\bar\mu}(\bar T)$ and $\int \bar\phi\,d\bar\mu$ (where $\bar\phi=\phi\circ\pi_X$)
valid for any ergodic invariant measure $\bar{\mu}$ on $P$ whose support contains non-trivial pin-sequences.
By using an ergodic decomposition of $\bar\nu$ the estimates obtained
for $h_{\bar{\mu}}(\bar T)$ and $\int\bar{\phi}\,d\bar{\mu}$ will immediately yield the same bounds
for $h_{\bar{\nu}}(\bar T)$ and $\int\bar{\phi}\,d\bar{\nu}$.

Let
\begin{equation}\label{PinSet}
  A=\{(x,z)\in P:z_0=1\},
\end{equation}
then $\bar\mu(A)>0$ and we can induce on $A$. Denote by $\bar\mu_A$
the restriction of $\bar\mu$ to the set $A$ (normalized), i.e.
$$
\bar\mu_A=\frac{1}{\bar\mu(A)}\bar\mu|_A.
$$

Let $\tau(x,z)=\min\{j\ge 1: \bar T^j(x,z)\in A\}$ be the first return time to $A$.
We obtain an induced system $(A,\bar T_A,\bar\mu_A)$, where the induced
map $\bar T_A$ is given by $\bar T_A(x,z)=\bar T^{\tau(x,z)}(x,z)$.
By Abramov's formula we get $h_{\bar\mu}(\bar T)= \bar\mu(A)h_{\bar\mu_A}(\bar T_A)$.

We let $\cQ$ be the partition of $A$ into sets
$$
Q_j=\{(x,z)\in A: \tau(x,z)=j\}.
$$
Then $q_j=\bar\mu_A(Q_j)$ is the distribution of return times. We have
$\sum_j q_j=1$. Moreover, by Kac's lemma,
$$
\sum_j jq_j=\frac{1}{\bar\mu(A)}.
$$

We refine the first return time partition $\cQ$  according to which
$\cL(X_s)$ the cylinder set belongs to. The partition $\cRR$ is a sub-partition of $\cQ$
into sets
\begin{align*}
Q_{j,s}&=\Big\{(x,z)\in Q_j\colon x_0\ldots x_{j-1}\in\cL(X_s)\Big\}
\text{\quad for $s=1,\ldots, \lfloor\log j\rfloor-3$;}\\
Q_{j,0}&
=\Big\{(x,z)\in Q_j\colon x_0\ldots x_{j-1}\in\bigcup_{s=\lfloor \log j\rfloor-2}^\infty \cL(X_s)\Big\}.
\end{align*}
Then $\cP$ is a further refinement of $\cRR$ into cylinder sets (so that $Q_j$ is refined into
separate cylinder
sets of length $j$).
We let \begin{equation}\label{defqjs}
q_{j,s}=\bar\mu_A(Q_{j,s}).
\end{equation}
and estimate $h_{\bar\mu}(\bar T)$ and $\int \bar{\phi}\,d\bar{\mu}$
in terms of the $q_{j,s}$. The next lemma provides an estimate on the entropy of the measure.

\begin{lemma}\label{Entropy}
For any $\bar T_A$-invariant probability measure $\bar{\mu}_A$ on $A$ we have
  \begin{equation*}
  h_{\bar{\mu}_A}(\bar T_A)\le2+\sum_{j=1}^{\infty}\left(
  q_{j,0}(8+3\log j)+\sum_{s=1}^{\lfloor\log j\rfloor-3}q_{j,s}(jh_s+3\log j)\right),
\end{equation*}
where $h_s=h_{\rm top}(X_s,T)$ and the $q_{j,s}$ are defined as in \eqref{defqjs}.
\end{lemma}

\begin{proof} Since the only measure that we use in the proof of
this lemma is $\bar\mu_A$, we suppress it from the notation and write $h=h_{\bar\mu_A}(\bar T_A)$.
The expansivity of the map $T$ implies that the partition $\cP$ described above is generating under $\bar T_A$,
so that $h\le H_{\bar\mu_A}(\cP)$.
Recall that $H_{\bar\mu_A}(\cP)=H_{\bar\mu_A}(\cRR)+H_{\bar\mu_A}(\cP|\cRR)$, 
where $H_{\bar\mu_A}(\cP|\cRR)$ is the conditional entropy of $\cP$ given $\cRR$, i.e.
$$
H_{\bar\mu_A}(\cP|\cRR)=\sum_{R\in\cRR}\bar\mu_A(R)\left(-\sum_{B\in\cP}
\frac{\bar\mu_A(R\cap B)}{\bar\mu_A(R)}\log \frac{\bar\mu_A(R\cap B)}{\bar\mu_A(R)}\right).
$$
Letting $n_{j,s}$ denote the number of cylinder sets of length $j$ that form
$Q_{j,s}$, then the term in parentheses in the above expression is bounded above by $n_{j,s}$ so that,
$$
h\le \sum_{j=1}^\infty \left(
\sum_{s=0}^{\lfloor \log j\rfloor-3}
q_{j,s}(-\log q_{j,s}+\log n_{j,s})\right).
$$

We then write $-\log q_{j,s}\le 2\log j+(-\log q_{j,s}-2\log j)^+$, where $(.)^+$ 
denotes the positive part of the function in parentheses. Taking notice that 
$-\log q_{j,s}-2\log j\ge 0$ when $q_{j,s}\le \frac 1{j^2}$, we estimate
\begin{align*}
h&\le\sum_{j=1}^\infty \left(
\sum_{s=0}^{\lfloor \log j\rfloor-3}
q_{j,s}(2\log j+\log n_{j,s})
+\sum_{q_{j,s}\le \frac 1{j^2}} q_{j,s}(-\log q_{j,s}-2\log j)\right)\\
&=
\sum_{j=1}^\infty \left(
\sum_{s=0}^{\lfloor \log j\rfloor-3}
q_{j,s}(2\log j+\log n_{j,s})
+\sum_{q_{j,s}\le \frac 1{j^2}} \frac1{j^2}\left(j^2q_{j,s}\log \frac{1}{j^2q_{j,s}}\right)\right)
\end{align*}
Since $x\log \frac 1x\le \frac {1}{e\ln 2}$ on $[0,1]$, we obtain
\begin{equation}\label{EntropyEstimate}
\begin{split}
h&\le\sum_{j=1}^\infty \left(
\sum_{s=0}^{\lfloor \log j\rfloor-3}
q_{j,s}(2\log j+\log n_{j,s})
+\sum_{q_{j,s}\le \frac 1{j^2}} \frac1{(e\ln 2)j^2}\right)\\
&\le \sum_{j=1}^\infty \left(
\sum_{s=0}^{\lfloor \log j\rfloor-3}
q_{j,s}(2\log j+\log n_{j,s})\right)
+\sum_{j=1}^\infty \frac{1+\lfloor\log j\rfloor}{(e\ln 2)j^2}\\
&\le \sum_{j=1}^\infty \left(
\sum_{s=0}^{\lfloor \log j\rfloor-3}
q_{j,s}(2\log j+\log n_{j,s})\right)
+2
\end{split}
\end{equation}

Recall that $n_{j,s}=\card\cL_j(X_s)$. To bound $n_{j,s}$
for $s\le\lfloor\log j\rfloor-3$, we note that each such block can overlap
$1+\lfloor j/2^s\rfloor$ words from the generator $W_s$ and there are two choices
for each word. Finally, up to $2^s$ subsequent applications of the shift transformation
to the given block produce different blocks.
Therefore,
\begin{equation*}
  n_{j,s}\le 2\cdot2^{\lfloor j/2^s\rfloor}\cdot 2^s,
\end{equation*}
which implies that $\log n_{j,s}\le s+1+j/2^s=s+1+jh_s\le jh_s+\log j$.
For $s=0$, a similar estimate shows $n_{j,0}\le 160j$.
Without going through all the details, in $\cL_j(X_{\lfloor \log j\rfloor -2})$, 
there are up to 512 choices each occurring with 
$2^{\lfloor\log j\rfloor-2}$ shifts of elements, for a total of at most $128j$ elements.
We then deal similarly with $\cL_j(X_{\lfloor\log j\rfloor -1})$ and
$\cL_j(\bigcup_{s\ge\log j}X_s)$ and combine the estimates.

Combining this with \eqref{EntropyEstimate}, we conclude that
\begin{equation*}
h\le 2+\sum_{j=1}^{\infty}\left(q_{j,0}(3\log j+8)+
\sum_{s=1}^{\lfloor\log j\rfloor-3}q_{j,s}(jh_s+3\log j)\right).
\end{equation*}
\end{proof}

We now turn our attention to the integral $\int\bar{\phi}\,d\bar{\mu}$.
\begin{lemma}\label{Integral} Suppose $A\subset P$ is the set defined
in (\ref{PinSet}), $\phi$ is the potential defined in Lemma \ref{Def_phi} and
$\bar\phi:X\times Z\to \bR$ is given by $\bar\phi(x,z)=\phi(x)$.
Then for any ergodic $\bar T$-invariant probability measure $\bar\mu$ on $P$ with
$\bar{\mu}(A)\ne 0$ we have
  \begin{equation*}
  \int\bar{\phi}\, d\bar{\mu}\le \bar\mu(A)\sum_{j=1}^{\infty}
 j \left(-q_{j}\delta_j+\sum_{s=1}^{\lfloor\log j\rfloor-3}c_sq_{j,s}\right),
\end{equation*}
where $q_{j,s}=\bar\mu_A(Q_{j,s})$ as above.
\end{lemma}
\begin{proof}
Let $\bar\phi_A(x,z)=\sum_{i=0}^{\tau(x,z)-1}\phi(T^ix)$ so
that $\int\bar\phi\,d\bar\mu=\bar\mu(A)\int \bar\phi_A\,d\bar\mu_A$.
To estimate $\int_{Q_j}\bar\phi_A\,d\bar\mu_A$, we use the partition
$\cRR$ defined above
and bound $\bar\phi_A$ on $Q_{j,s}$.
For $(x,z)\in Q_{j,s}$ we know that the pins are at zero and $j$ and the
block $x_{[0,j-1]}$ is in the language of $X_s$. The former implies
that $x_{[0,j]}\notin \cL(Y)$, so
\begin{equation}\label{eq:distbound}
\dist (T^ix,Y)\ge 2^{-j} \text{ for $i=0,\ldots,j-1$.}
\end{equation}
We claim that if $(x,z)\in Q_{j,s}$ with $s\le\lfloor\log j\rfloor-3$
then
$\phi(T^i(x))$ is bounded above by $c_s-\delta_j$.
For $m\le s$ we get from \eqref{eq:distbound} that $\phi_m(T^i(x))\le
c_m-\delta_j\le c_s-\delta_j$ since the $c_m$ are increasing. We now deal with
the case $m>s$. By assumption, $j\ge 2^{s+3}$, so since $x_{[0,j-1]}\in\cL_j(X_s)$
has occurrences of the last two symbols of a word in $W_s$ (that is either 01 or 10),
every $2^s$ symbols. In particular, in any block of $2^{s+1}+1$ symbols, there
are two such transitions, so that no block of length $2^{s+1}+1$ lies in $\cL(X_m)$
for any $m>s$,
and hence for $i=0,\ldots,j-1$, $\phi_m(T^i(x))\le c_m-\delta_{2^{s+1}+1}\le 0-\delta_{2^{s+2}}$.
It suffices to make sure that $-\delta_{2^{s+2}}\le c_s-\delta_{j}$ for $s\le \lfloor \log j\rfloor -3$.
This is true because
$$
\delta_{2^{s+2}}-\delta_j\ge \delta_{2^{s+2}}-\delta_{2^{s+3}}
=\frac{13+3s}{2^{s+3}\beta_1}\ge \frac{1}{2^s\beta_1}> -c_s.
$$
Hence we see that for $i=0,\ldots,j-1$, $\bar\phi(\bar T^i(x,z))=\phi(T^i(x))$
is bounded by $c_s-\delta_{j}$ for $(x,z)\in Q_{j,s}$ for $s\le\lfloor \log j\rfloor -3$ and by $-\delta_{j}$
for $(x,z)\in Q_{j,0}$.

We  obtain
\begin{align*}
  \int_{Q_j}\bar\phi_A\,d\bar\mu_A &
  =\sum_{s=1}^{\lfloor \log j\rfloor-3}\int_{Q_{j,s}}\bar\phi_A \,d\bar\mu_A +\int_{Q_{j,0}}
  \bar\phi_A \,d\bar\mu_A\\
   & \le -jq_j\delta_j+\sum_{s=1}^{\lfloor \log j\rfloor -3}\int_{Q_{j,s}}jc_s \,
   d\bar\mu_A \\
   &= j\left(-q_j\delta_j + \sum_{s=1}^{\lfloor \log j\rfloor -3} q_{j,s}c_s\right),
\end{align*}
and the proof is complete.
\end{proof}
We are now ready to complete the proof of Theorem \ref{main}.

\begin{proof}[Proof of Theorem \ref{main}]
We shall show that given any ergodic $T$-invariant measure $\mu$  and any 
$\beta\in [\beta_{n-1},\beta_n]$ for $n>1$ or $\beta\in [\alpha,\beta_1]$ for $n=1$ we have
\begin{equation}\label{eq:toshow}
h_\mu(T)+\int \beta\phi\,d\mu\le h_{\mu_n}+\int\beta\phi\,d\mu_n=h_n+\beta c_n,
\end{equation}
with equality if and only if $\mu$ is either $\mu_n$ or $\mu_{n+1}$. Any ergodic measure is
supported on some $X_k$, or $X_\infty$, or $X\setminus Y$. Among measures supported on $X_k$,
the maximum of $h_\mu(T)+\beta\int\phi\,d\mu$ is achieved by $\mu_k$ so that
\eqref{eq:toshow} holds for measures supported on $\bigcup X_k$ by Lemma \ref{Choice of c_n}.
For measures supported on $X_\infty$, $h_\mu(T)=0$ and $\int\phi\,d\mu=c=0$, so that
again Lemma \ref{Choice of c_n} establishes \eqref{eq:toshow} in this case.

We now consider the case when $\mu$ is an ergodic measure supported on $X\setminus Y$, and we let
$\bar\mu$ be an ergodic measure on $P$ such that ${\pi_X}_*\bar\mu=\mu$. 
Recall that $h_\mu(T)=h_{\bar\mu}(\bar T)
=\bar\mu(A)h_{\bar\mu_A}(\bar T_A)$, where $A$ is the pin set.
From Lemmas \ref{Entropy} and \ref{Integral}, we see
\begin{align*}
\frac{1}{\bar\mu(A)}\left(h_{\bar\mu}(\bar T)+\int\beta\bar\phi\,d\bar\mu\right)
\le\;&
2+
\sum_{j=1}^\infty\sum_{s=1}^{\lfloor\log j\rfloor-3}
q_{j,s}(jh_s+3\log j+\beta jc_s)\\
&+\sum_{j=1}^\infty q_{j,0}(8+3\log j)-\beta\sum_{j=1}^\infty jq_j\delta_j
\end{align*}
Using the inequalities $h_s+\beta c_s\le h_n+\beta c_n$ from Lemma \ref{Choice of c_n}, we have
\begin{align*}
\frac{1}{\bar\mu(A)}\left(h_{\bar\mu}(\bar T)+\int\beta\bar\phi\,d\bar\mu\right)\le\;&
2+(h_n+\beta c_n)\sum_{j=1}^\infty \sum_{s=0}^{\lfloor\log j\rfloor-3}q_{j,s}j\\
&+\sum_{j=1}^\infty(8+3\log j)\sum_{s=0}^{\lfloor\log j\rfloor-3}q_{j,s}-
\beta\sum_{j=1}^\infty jq_j\delta_j\\
\le\; &2+\frac{h_n+\beta c_n}{\bar\mu(A)}+\sum_{j=1}^\infty q_j(8+3\log j-j\beta\delta_j).
\end{align*}
Since $j\delta_j\beta>j\delta_j\alpha=10+3\log j$ and $\sum_{j=1}^\infty q_j=1$,
we see
$$
\frac{1}{\bar\mu(A)}\left(h_{\bar\mu}(T)+\int\beta\bar\phi\,d\bar\mu\right)\le
\frac{h_n+\beta c_n}{\bar\mu(A)}
$$
as required.
\end{proof}

\begin{remark}\label{NoMoreTransitions}
Note that on the interval $[\alpha,\infty)$ there are no phase transitions other 
than at points $\beta_n$. In addition, $\mu_{n+1}$ is an equilibrium state for 
$\beta\phi$ for all $\beta\in [\beta_n,\beta_{n+1}]$.

\end{remark}

Next, we address the case of finitely many phase transitions.
\begin{corollary}\label{FiniteCase}
 For any set of positive real numbers $\alpha<\beta_1<\beta_2<\dots<\beta_N$  
 there exists a continuous potential $\phi:X\to\bR$ such that for $\beta\ge\alpha$ 
 the only points of non-differentiability of the pressure function 
 $\beta\mapsto P_{\rm top}(\beta\phi)$ are $\beta_1,...,\beta_N$.
\end{corollary}
\begin{proof}
Our construction can be easily modified by truncating it at $n=N$ as follows. For $n=1,..., N+1$ we consider subshifts $(X_n)$  with generators $W_n=\{(\underbrace{00...01}_{2^n}),(\underbrace{11...10}_{2^n})\}$. We let $c_1$ be any real number, and define
$$
c_{n+1}=c_1+\sum_{j=1}^n\frac{1}{2^{j+1}\beta_j},\quad \text{for }n=1,...,N.
$$
One can check that for any integer $n\le N$ we have $c_n<c_{n+1}$, and the 
conditions similar to the ones in Lemma \ref{Choice of c_n} hold, i.e.
\begin{enumerate}
  \item $h_{n}+\beta_n c_{n}=h_{n+1}+\beta_n c_{n+1}$.
  \item For $k\le N+1$ we have $h_{k}+\beta c_{k}\le h_{n}+\beta c_{n}$ whenever 
   $\beta\in[\beta_{n},\beta_{n+1}]$ and $h_k+\beta c_k\le h_1+\beta c_1$ whenever 
   $\beta\in[0,\beta_1]$ . The inequalities are strict as long as $k\not\in\{n,n+1\}$.
  \item $\beta c_{N+1}\le h_{n}+\beta c_{n}$ for $0\le\beta\le\beta_1$ and $n=1,...,N$.
\end{enumerate}

Analogously to Lemma \ref{Def_phi} we define  $\phi(x)=\max\{\phi_n(x):\,1\le n\le N+1\}$, where
\begin{equation*}
 \phi_n(x)= c_n-\delta_{j},\text{ if }\,{\rm dist}(x, X_n)=\frac{1}{2^j}\quad
 \text{with}\,\,\delta_j=\frac{10+3\log j}{\alpha j}.
\end{equation*}
The continuity of $\phi$ is now immediate, since it is a maximum of finitely many 
continuous functions. The other characteristics of $\phi$ asserted in Lemma 
\ref{Def_phi} follow from properties of $(\beta_n)_{n=1}^N$ and $(c_n)_{n=1}^{N+1}$ 
above. The rest of the proof is a verbatim repetition of the proof of Theorem \ref{main} 
with the assumption that $q_{s,j}=0$ whenever $s>N+1$.
\end{proof}

We conclude with a discussion of the freezing properties of the potential $\phi$.
Define $I(\phi)=\{\int \phi\, d\mu: \mu\in \cM\}=[a_\phi,b_\phi]$. For 
$t\in I(\phi)$ we define the localized entropy
at $t$ by $\cH(t)=\sup\{h_\mu(T): \int \phi\, d\mu=t\}$. It follows that $\cH$ is a 
concave and upper semi-continuous (and hence continuous) function.
The proofs of the freezing properties of $\phi$ are based on the following 
elementary facts which are immediate consequences of the Variational Principle 
\eqref{VarPrinciple} and the upper semi-continuity of the
entropy map.
\begin{enumerate}
\item If $0<\beta_1<\beta_2$ and $\mu_i$ is an equilibrium state for $\beta_i\phi$ 
then $\int \phi\, d\mu_1\leq \int \phi\, d\mu_2$ and $h_{\mu_2}(T)\leq h_{\mu_1}(T)$.
\item If $\mu$ is an equilibrium state for $\tilde{\beta}\phi$ with 
 $\int \phi\, d\mu=b_\phi$ then for any $\beta>\tilde{\beta}$ a measure $\nu$ is an 
 equilibrium state for $\beta\phi$ provided that $\smallint \phi\,d\nu=b_\phi$ and 
 $h_\nu(T)=\cH(b_\phi)$.
\item If $\lim\limits_{n\to \infty} \beta_n=\beta_\infty<\infty$ and 
 $\mu=\lim\limits_{n\to \infty} \mu_n$ where $\mu_n$ is an equilibrium state of 
 $\beta_{n}\phi$ then   $\int \phi\, d\mu=\lim\limits_{n\to \infty} \int \phi\, d\mu_n$ 
 and $\mu$  is an equilibrium state for $\beta_{\infty}\phi$.
\end{enumerate}

\begin{corollary}\label{Quasi-crystal}
  If $\lim\limits_{n\to\infty} \beta_n=\beta_\infty<\infty$, then  the family of 
  equilibrium states of $\beta\phi$ is constant for all $\beta>\beta_\infty$.
\end{corollary}
\begin{proof}
Let $\mu_x$ and $\mu_{y}$ denote the Dirac measures supported on the 
fixed points $x=(\bar{0})$ and $y=(\bar{1})$ respectively. Further, let $\mu_n$ 
denote the unique measure of maximal entropy of subshift $X_n$. Recall that by 
construction $\int \phi\, d\mu_n =c_n$ and $\lim\limits_{n\to \infty} c_n=c$.
We observe that $b_\phi=c$ and
\begin{equation}\label{eqqc}
  \left\{\mu\in \cM^e: \smallint \phi\, d\mu=c\right\}=\{\mu_x,\mu_y\}.
\end{equation}
If $\lim\limits_{n\to\infty} \beta_n = \beta_{\infty}<\infty$ then, by applying 
property 3 above to a weak$^\ast$ accumulation point $\mu$ of $(\mu_n)$, 
we conclude that $\int_X \phi\, d\mu=c$ and $\mu$ is an equilibrium state for 
$\beta_\infty\phi$. Finally,  property 2 and \eqref{eqqc}  imply that for all 
$\beta>\beta_\infty$ the set of equilibrium states for the potential 
$\beta \phi$ is precisely   $\{\alpha\mu_x+(1-\alpha)\mu_y: \alpha\in [0,1]\}$.
\end{proof}

Lastly, we attend to the freezing property in the case of finitely many discontinuity points $\beta_n$.
\begin{corollary} \label{cor123}
Let $\beta_1,\dots,\beta_{N+1}$ and $\phi$ be as in Corollary \ref{FiniteCase} 
and let $\mu_{N+1}$ denote the measure of maximal entropy on $X_{N+1}$. 
Then $\mu_{N+1}$ is the unique equilibrium state of $\beta\phi$ for all $\beta>\beta_N$.
\end{corollary}
\begin{proof}
The proof is analogous to the proof of Corollary \ref{Quasi-crystal} and is left to the reader.
\end{proof}

\end{document}